\documentclass[12pt]{amsproc}
\newtheorem{theorem}{\sc Theorem}
\newtheorem{lemma}[theorem]{\sc Lemma}

\newtheorem{corollary}[theorem]{\sc Corollary}

\begin{document} 
\author{ Raimundo Bastos}
\address{ Departamento de Matem\'atica, Universidade de Bras\'ilia,
Bras\'ilia-DF, 70910-900 Brazil }
\email{bastos@mat.unb.br}
\author{ Carmine Monetta }
\address{Dipartimento di Matematica, Universit\`a di Salerno, Via Giovanni Paolo II, 132 - 84084 - Fisciano (SA), Italy}
\email{cmonetta@unisa.it}
\author{ Pavel Shumyatsky }
\address{ Departamento de Matem\'atica, Universidade de Bras\'ilia,
Bras\'ilia-DF, 70910-900 Brazil}
\email{pavel@unb.br}


\title[metanilpotency]{A criterion for metanilpotency of a finite group}

\maketitle

\begin{abstract}
We prove that the $k$th term of the lower central series of a finite group $G$ is nilpotent if and only if $|ab|=|a||b|$ for any $\gamma_k$-commutators $a,b\in G$ of coprime orders.  
\end{abstract}

\section{Introduction} 
All groups considered in this article are finite. 
The following sufficient condition for nilpotency of a group $G$ was discovered by B. Baumslag and J. Wiegold \cite{baubau}.
\medskip

{\it Let $G$ be a group in which $|ab|=|a||b|$ whenever the elements $a,b$ have coprime orders. Then $G$ is nilpotent.}
\medskip

Here the symbol $|x|$ stands for the order of the element $x$ in a group $G$. In \cite{BS} a similar sufficient condition for nilpotency of the commutator subgroup $G'$ was established. 
\medskip

{\it Let $G$ be a group in which $|ab|=|a||b|$ whenever the elements $a,b$ are commutators of coprime orders. Then $G'$ is nilpotent.}
\medskip

Of course, the conditions in both above results are also necessary for the nilpotency of $G$ and $G'$, respectively. In the present article we extend the above results as follows.

Given an integer $k\geq1$, the word $\gamma_{k}=\gamma_k(x_1,\dots,x_k)$ is defined inductively by the formulae
\[
\gamma_1=x_1,
\qquad \text{and} \qquad
\gamma_k=[\gamma_{k-1},x_k]=[x_1,\ldots,x_k]
\quad
\text{for $k\ge 2$.}
\]
The subgroup of a group $G$ generated by all values of the word $\gamma_k$ is denoted by $\gamma_k(G)$. Of course, this is the familiar $k$th term of the lower central series of $G$. If $k=2$ we have $\gamma_k(G)=G'$. In the sequel the values of the word $\gamma_k$ in $G$ will be called $\gamma_k$-commutators.

\begin{theorem}\label{main} The $k$th term of the lower central series of a group $G$ is nilpotent if and only if $|ab|=|a||b|$ for any $\gamma_k$-commutators $a,b\in G$ of coprime orders. 
\end{theorem}

Recall that a group $G$ is called metanilpotent if there is a normal nilpotent subgroup $N$ such that $G/N$ is nilpotent. The following corollary is immediate.

\begin{corollary}\label{ain} A group $G$ is metanilpotent if and only if there exists a positive integer $k$ such that $|ab|=|a||b|$ for any $\gamma_k$-commutators $a,b\in G$ of coprime orders. 
\end{corollary}

We suspect that a similar criterion of nilpotency of the $k$th term of the derived series of $G$ can be established. On the other hand, Kassabov and Nikolov showed in \cite{kani} that for any $n\geq7$ the alternating group $A_n$ admits a commutator word all of whose nontrivial values have order 3. Thus, the verbal subgroup $w(G)$ need not be nilpotent even if all $w$-values have order dividing 3.

\section{Proofs}
As usual, if $\pi$ is a set of primes, we denote by $\pi'$ the set of all primes that do not belong to $\pi$. For a group $G$ we denote by $\pi(G)$ the set of primes dividing the order of $G$. The maximal normal $\pi$-subgroup of $G$ is denoted by $O_{\pi}(G)$. The Fitting subgroup of $G$ is denoted by $F(G)$. The Fitting height of $G$ is denoted by $h(G)$. Throughout the article we use without special references the well-known properties of coprime actions:  if $\alpha$ is an automorphism of a finite group $G$ of coprime order, $(|\alpha|,|G|)=1$, then $C_{G/N}(\alpha)=C_G(\alpha)N/N$ for any $\alpha$-invariant normal subgroup $N$, the equality $[G,\alpha]=[[G,\alpha],\alpha]$ holds, and if $G$ is in addition abelian, then $G=[G,\alpha]\times C_G(\alpha)$. Here $[G,\alpha]$ is the subgroup of $G$ generated by the elements of the form $g^{-1}g^\alpha$, where $g\in G$.

For elements $x,y$ of a group $G$ write $[x,{}_0y]=x$ and $[x,{}_{i+1}y]=[[x,{}_{i}y],y]$ for $i\geq0$. An element $y\in G$ is called Engel if for any $x\in G$ there is a positive integer $n=n(x)$ such that $[x,{}_{n}y]=1$.

The following lemma is well-known. 
\begin{lemma} \label{meta} Let $p$ be a prime and $G$ a metanilpotent group.  Suppose that $x$ is a $p$-element in $G$ such that $[O_{p'}(F(G)),x]=1$. Then $x\in F(G)$. 
\end{lemma}

\begin{proof} Since all Engel elements of a finite group lie in the Fitting subgroup \cite[12.3.7]{Rob}, it is sufficient to show that $x$ is an Engel element. Let $F=F(G)$ and $P$ be the Sylow $p$-subgroup of $F$. We have $F= P\times O_{p'}(F)$. By hypothesis, $G/F$ is nilpotent of class $n$ for some positive integer $n$. We deduce that $[G,{}_{n}x]\leq F$ and so $[G,_{n+1}x] \leq P$. Therefore $\langle[G,_{n+1}x],x\rangle$ is contained in a Sylow $p$-subgroup of $G$. Hence, $x$ is an Engel element in $G$ and so the lemma follows. 
\end{proof}

\begin{lemma}\label{gamma} Let $k$ be a positive integer and $G$ a group such that $G=G'$. Let $q\in\pi(G)$. Then $G$ is generated by $\gamma_k$-commutators of $p$-power order for primes $p\neq q$.  
\end{lemma}
\begin{proof} For each prime $p\in\pi(G)\setminus\{q\}$ let $N_p$ denote the subgroup generated by all $\gamma_k$-commutators of $p$-power order. Let us show first that for each $p$ the Sylow $p$-subgroups of $G$ are contained in $N_p$. Suppose that this is false and choose $p$ such that a Sylow $p$-subgroup of $G$ is not contained in $N_p$. We can pass to the quotient $G/N_p$ and assume that $N_p=1$. Since $G=G'$, it is clear that $G$ does not possess a normal $p$-complement. Therefore the Frobenius Theorem \cite[Theorem 7.4.5]{go} shows that $G$ has a $p$-subgroup $H$ and a $p'$-element $a\in N_G(H)$ such that $[H,a]\neq1$. We have $$1\neq[H,a]=[H,\underbrace{a,\ldots,a}_{(k-1)\ times}]\leq N_p,$$ a contradiction. Therefore indeed $N_p$ contains the Sylow $p$-subgroups of $G$. Let $T$ be the product of all $N_p$ for $p\neq q$. We see that $G/T$ is a $q$-group. Since $G=G'$, we conclude that $G=T$. The proof is complete.
\end{proof}

Let us call a subgroup $H$ of $G$ a tower of height $h$ if $H$ can be written as a product $H=P_1\cdots P_h$, where

(1) $P_i$ is a $p_i$-group ($p_i$ a prime) for $i=1,\dots,h$.

(2) $P_i$ normalizes $P_j$ for $i<j$.

(3) $[P_i,P_{i-1}]=P_i$ for $i=2,\dots,h$. \\ 

It follows from (3) that $p_i\neq p_{i+1}$ for $i=1,\dots,h-1$. A finite soluble group $G$ has Fitting height at least $h$ if and only if $G$ possesses a tower of height $h$ (see for example \cite{turull}). \\

Throughout the remaining part of the article $G$ denotes a finite group for which there exists $k\geq1$ such that $|ab|=|a||b|$ for any $\gamma_k$-commutators $a,b\in G$ of coprime orders. We denote by $X$ the set of all $\gamma_k$-commutators in $G$.

\begin{lemma}\label{bbb} Let $x\in X$ and $N$ be a subgroup normalized by $x$. If $(|N|,|x|)=1$, then $[N,x]=1$.
\end{lemma}
\begin{proof} Choose $y\in N$. The order of the $\gamma_k$-commutator $[x,y]$ is prime to that of $x$. Therefore we must have $|x[x,y]|=|x||[x,y]|$. However $x[x,y]=y^{-1}xy$. This is a conjugate of $x$ and so $|x[x,y]|=|x|$. Therefore $[x,y]=1$.
\end{proof}

\begin{lemma}\label{solu} If $G$ is soluble, then the subgroup $\gamma_k(G)$ is nilpotent.
\end{lemma}
\begin{proof} Set $h=h(G)$ and $F = F(G)$. If $G$ is nilpotent, the result is immediate so we assume that $h\geq 2$.

We first examine the case $h=2$. If $G/F$ has nilpotency class at most $k-1$, then $\gamma_k(G) \leq F$ is nilpotent. So we will assume that $G/F$ is nilpotent of class at least $k$. Hence, the image of some Sylow $p$-subgroup $P$ of $G$ in $G/F$ has nilpotency class at least $k$. Therefore, there exists a $\gamma_k$-commutator $x$ in elements of $P$ which does not belong to $F$. By Lemma \ref{bbb} $[O_{p'}(F),x]=1$, whence by Lemma \ref{meta}, $x\in F$. This is a contradiction.

Now assume that $h\geq 3$. By \cite[Lemma 1.9]{turull}, there exists a tower $P_1P_2P_3\ldots P_h$ of height $h$ in $G$. Since $P_2=[P_2,P_1]$, it follows that $$P_2=[P_2,\underbrace{P_1, \ldots,P_1}_{(k-1)\ times}].$$ 
Combining Lemma \ref{bbb} with the fact that $P_2$ is generated by $\gamma_k$-commutators of $G$ of $p_2$-orders, we deduce that $P_3$ commutes with $P_2$. On the other hand, $[P_3,P_2]=P_3$, because $P_1P_2P_3\ldots P_h$ is a tower. This is a contradiction.  The proof is complete. 
\end{proof}

We are now in a position to complete the proof of Theorem \ref{main}.

\begin{proof}[Proof of Theorem \ref{main}] It is clear that if $\gamma_k(G)$ is nilpotent, then $|ab|=|a||b|$ for any $\gamma_k$-commutators $a,b\in G$ of coprime orders. So we only need to prove the converse. Since the case where $k\leq2$ was considered in \cite{BS} and \cite{baubau}, we will assume that $k\geq3$.

Suppose that the theorem is false and let $G$ be a counterexample of minimal order. In view of Lemma \ref{solu} $G$ is not soluble while all proper subgroups of $G$ are. Therefore $G=G'$. Let $R$ be the soluble radical of $G$. It follows that $G/R$ is nonabelian simple. By Lemma \ref{solu} $\gamma_k(R)$ is nilpotent. Suppose that $R\neq1$. 

Choose $q\in\pi(F(G))$. According to Lemma \ref{gamma} $G$ is generated by the $\gamma_k$-commutators of $p$-power order for primes $p\neq q$. Let $Q$ be the Sylow $q$-subgroup of $F(G)$. By Lemma \ref{bbb}, $[Q,x]=1$, for every $\gamma_k$-commutator $x$ of $q'$-order. Therefore $Q\leq Z(G)$. This happens for each choice of $q\in\pi(F(G))$ so we conclude that $F(G)\leq Z(G)$.  

For each $x\in R$ and $y\in G$ we have $$[y,\underbrace{x,\ldots,x}_{k\ times}]\in\gamma_k(R)\leq F(G)=Z(G).$$ Consequently, every element $x\in R$ is Engel. Therefore $R\leq F(G)$ (\cite[12.7.4]{Rob}). Hence, $R=Z(G)$ and $G$ is quasisimple.   

Since $G$ does not possess a normal 2-complement, it follows from the Frobenius Theorem \cite[Theorem 7.4.5]{go} that $G$ contains a 2-subgroup $H$ and an element of odd order $b\in N_G(H)$ such that $[H,b]\neq1$. In view of Thompson's Theorem \cite[Theorem 5.3.11]{go} we can assume that $H$ is of nilpotency class at most two and $H/Z(H)$ is elementary abelian. We claim that $G$ contains an element $a\in X$ such that $a$ is a 2-element and $a$ has order 2 modulo $Z(G)$. 

Recall that $k\geq3$. Let $Y$ be the set of all elements $[h,{}_{(k-2)}b]$, where $h\in H$. Since $H$ is nilpotent of class at most $2$, it follows that for every $y\in Y$ all elements in $[H,y]$ are $\gamma_k$-commutators. If for some $y\in Y$ the subgroup $[H,y]$ does not lie in $Z(G)$, every element of $[H,y]$ having order $2$ modulo $Z(G)$ enjoys the required properties. Therefore we assume that $[H,y]\leq Z(G)$ for every $y\in Y$. Let $K$ be the subgroup generated by $Y$. Obviously, $K=[K,b]$ and $K'\leq Z(G)$. It follows that $C_K(b)\leq Z(G)$ and all elements in $K$ are $\gamma_k$-commutators modulo $Z(G)$. If $d\in K$ such that $d\not\in Z(G)$ and $d^2\in Z(G)$, then the commutator $[d,b]$ is as required.

Now we fix an element $a$ with the above properties. Since $G/Z(G)$ is nonabelian simple, it follows from the Baer-Suzuki Theorem \cite[Theorem 3.8.2]{go} that there exists an element $t\in G$ such that the order of $[a,t]$ is odd. On the one hand, it is clear that $a$ inverts $[a,t]$. On the other hand, by Lemma \ref{bbb}, $a$ centralizes $[a,t]$. This is a contradiction.
\end{proof}
\section{Acknowledgment}
The work of the first and the third authors was supported by CNPq-Brazil and FAPDF. The second author was partially supported by the National Group for Algebraic and Geometric Structures, and their Applications  (GNSAGA -- INdAM). This article was written during the second author's visit to the University of Brasilia. He wishes to thank the Department of Mathematics for excellent hospitality.

\end{document}